\documentclass[a4paper, reqno, draft]{amsart}
\usepackage{amssymb,latexsym}
\usepackage{amsmath}
\usepackage{amsthm}
\usepackage{graphicx}
\usepackage{hyperref}
\usepackage{titletoc}
\usepackage{palatino,mathpazo}
\usepackage{amsthm, amsmath, amssymb, amsfonts, amsopn, color, mdwlist, geometry}
\usepackage{enumerate}
\numberwithin{equation}{section}
\newtheorem{theorem}{Theorem}[section]
\newtheorem{proposition}{Proposition}[section]
\newtheorem{lemma}{Lemma}[section]

\theoremstyle{definition}

\def\XXint#1#2#3{{\setbox0=\hbox{$#1{#2#3}{\int}$}
     \vcenter{\hbox{$#2#3$}}\kern-.5\wd0}}

\def\e{\varepsilon}

\def\R{{\mathbb R}}

\def\RN{\mathbb{R}^2}
\def\ulm{{u}_{1,n}}

\def\lm0{ n >0}
\def\pl{\phi_{1, n }}
\def\nl{{u}_{2,n}}
\def\wlm{w_{1, n }}
\def\wlmt{w_{2, n }}
\def\R{\mathbb{R}}
\def\RN{\mathbb{R}^2}

\def\e{\varepsilon}

\begin{document}

\title[The asymptotic behavior of Chern-Simons Vortices for Gudnason Model]{The asymptotic behavior of Chern-Simons Vortices for Gudnason Model}

\author{Youngae Lee}
\address[Youngae Lee]
{Department of Mathematics Education, Teachers College, Kyungpook National University, Daegu, South Korea}
\email{youngaelee@knu.ac.kr}

\begin{abstract}
We consider an elliptic system arising from  a  supersymmetric gauge field theory. In this paper, we complete to classify all possible solutions according to their asymptotic  behavior under a weak coupling effect. Interestingly, it turns out that one of components does not follow the   feature of  condensate solutions
for the classical Chern-Simons-Higgs model.  Moreover, in order to prove the concentration property of blow up component, we need to improve the convergence rate  and the gradient estimation for  the other component, which converges to a constant.  We expect that this result would provide an   insight for the study of general elliptic system problems, which are even neither cooperative nor competitive.
\end{abstract}

\date{\today}
\keywords{blow up analysis; asymptotic behavior; Pohozaev Identity; doubly periodic solution; 35B40; 35J20}

\maketitle

\section{Introduction}\label{sec1}
 In various areas of physics,    solitons play an important role as static  solutions of  an Euler-Largrange equation carrying finite energy. The  bi-dimensional soliton solutions (vortices)  have been   interdisciplinary topics between particle physics and condensed matter physics (see \cite{A,     CN, CS1,  H,  JZ,    PM, R} and references therein). For  the study of  high temperature superconductivity \cite{KF, MNK}, the Bose-Einstein condensates \cite{IGRCGGLPK,KO}, the quantum Hall effect \cite{So}, optics \cite{BEC}, and superfluids \cite{Sh}, we need to consider  the dyonic vortices, which are   electrically and magnetically charged vortices  (see also \cite{ B, DGP, D, DJPT, GL,  HKP,  JT, JW, Ta} and references therein).  In order to study the coexistence of electric and magnetic charges, Chern-Simons theories have been extensively developed (see \cite{ BCCT, BT, CY, CI, CFL,   CHMY, CK,  DEFM, H2, HKP,  HT, JTa,  LY1, LY3, NT1, NT2, NT3, JW, SYa1, SYa2,  T1,T3, W, Y3,Y4} and references therein).
Among them, we are mainly interested in    non-Abelian Chern-
Simon-Higgs vortices in the Gudnason model. In \cite{G1, G2}, the Gudnason model  was proposed to  investigate   a supersymmetric
gauge field theory  (see also \cite{ ABJM, ABEKY,  BL2, Gu,  HSZ, HTo,MY,SY1, SY2} and references therein for the backgrounds).    In this paper, we are concerned with the following elliptic system, which  can be reduced from   the Gudnason model (see \cite{G1, G2, HLTY} for the detail):
\begin{equation}\label{eq0}
\left\{\begin{array}{l}
\Delta {u}_{1,n}=(\alpha_n+\beta_n)^2 e^{{u}_{1,n}}(e^{{u}_{1,n}}-1)+(\alpha_n-\beta_n)^2 e^{{u}_{2,n}}(e^{{u}_{1,n}}-1)-(\beta_n^2-\alpha_n^2) (e^{{u}_{1,n}}+e^{{u}_{2,n}})(e^{{u}_{2,n}}-1)\\\quad\quad\ +8\pi\sum_{i=1}^N m_{i}\delta_{p_i},\\ \\
\Delta {u}_{2,n}=(\alpha_n+\beta_n)^2 e^{{u}_{2,n}}(e^{{u}_{2,n}}-1)+(\alpha_n-\beta_n)^2 e^{{u}_{1,n}}(e^{{u}_{2,n}}-1)-(\beta_n^2-\alpha_n^2) (e^{{u}_{1,n}}+e^{{u}_{2,n}})(e^{{u}_{1,n}}-1),
\end{array}
\right.
\end{equation}
where   $\alpha_n, \beta_n>0$ are positive parameters,   $\delta_{p_i}\in \mathbb{T}$ stands for the
Dirac measure concentrated at $p_i$, and $p_i\neq p_j$ if $i\neq j$.  Each $p_i$ is called a vortex point  and  $m_i\in\mathbb{N}$ is  the multiplicity of $p_i$.   Due to the theory suggested by ’t Hooft  in \cite{'tH}, we can consider the equation \eqref{eq0} in a flat two torus $\mathbb{T}$.  Indeed, the periodic boundary
condition is motivated by the  lattice structures forming in a
condensed matter system (for example, see \cite{A, T3}).

Recently, Han, Lin, Tarantello, and Yang in \cite{HLTY} showed  that there exist at least two gauge-distinct solutions  of \eqref{eq0}  in   $\mathbb{T}$.  The authors in \cite{HLTY} also established the existence  result of solutions  for \eqref{eq0} in $\mathbb{R}$, which vanish at infinity, with  an arbitrary number of components $u_{i, n}$, $i=1,\cdots, m$,  $m\ge 2$. In  order to obtain the existence result, they applied   a minimization approach. For the recent developments of \eqref{eq0},  we refer the readers to \cite{CHLS, CL1,CL2, CKL, GGW, HH, HT, HT1, HL, LY2, PT, T2}.

Our main goal in this paper is to classify all possible solutions according to their asymptotic  behavior  as $\alpha_n, \beta_n\to \infty$.
Throughout this paper, we assume that \begin{equation}\label{assume1}\alpha_n>0,\ \ \beta_n>0\ \  \textrm{for all}\ n\ge 1,\ \textrm{and}\ \lim_{n\to\infty}(\alpha_n+\beta_n)=+\infty.\end{equation}
 In order to  carry out blow up analysis, we firstly need to obtain the uniform boundedness of  $L^1$ norm  for nonlinear terms in \eqref{eq0} with respect to $\alpha_n, \beta_n>0$. However,  if $|\beta_n-\alpha_n|$ is not small enough, then  the system \eqref{eq0} might be  neither cooperative nor competitive, and it causes a difficulty to derive the uniform boundedness of  $L^1$ norm  for nonlinear terms.  In order to avoid this difficulty, we assume the weak coupling effect such that  there is a constant $\mathfrak{N}>0$ satisfying
\begin{equation}\label{assume2}0<(\beta_n-\alpha_n)(\alpha_n+\beta_n)\le \mathfrak{N}\quad \textrm{for}\quad n\ge 1.\end{equation}
In view of  the conditions  \eqref{assume1} and \eqref{assume2}, we note that \begin{equation}\label{assume3}\lim_{n\to\infty}\left(\frac{\beta_n-\alpha_n}{\alpha_n+\beta_n}\right)=0.\end{equation}
For the simplicity, let \begin{equation}\label{epsilon}\varepsilon_n:=\frac{1}{\alpha_n+\beta_n},\quad \sigma_n:=\frac{\beta_n-\alpha_n}{\alpha_n+\beta_n}\end{equation}
We can rewrite \eqref{eq0} as follows:
\begin{equation}\label{main_eq}
\left\{\begin{array}{l}
\Delta {u}_{1,n}=\frac{1}{\varepsilon_n^2} \left\{e^{{u}_{1,n}}(e^{{u}_{1,n}}-1)+\sigma_n^2 e^{{u}_{2,n}}(e^{{u}_{1,n}}-1)-\sigma_n (e^{{u}_{1,n}}+e^{{u}_{2,n}})(e^{{u}_{2,n}}-1)\right\}\\\quad\quad\ +8\pi\sum_{i=1}^N m_{i}\delta_{p_i},\\ \\
\Delta {u}_{2,n}=\frac{1}{\varepsilon_n^2}\left\{e^{{u}_{2,n}}(e^{{u}_{2,n}}-1)+\sigma_n^2 e^{{u}_{1,n}}(e^{{u}_{2,n}}-1)-\sigma_n (e^{{u}_{1,n}}+e^{{u}_{2,n}})(e^{{u}_{1,n}}-1)\right\},
\end{array}
\right.
\end{equation}
%In view of Theorem \ref{BrezisMerletypealternatives} above, the equation \eqref{main_eq}  has different kinds of periodic solutions
%satisfying one of the following asymptotic behaviors:\begin{equation}
%\label{def_sol}
%\begin{aligned}
  %\left({u}_{1,n},{u}_{2,n}\right) \to (0,0) \quad\mbox{a.e. on }~ \mathbb{T} \quad\mbox{as }~ {n}\to\infty,
% \quad&\mbox{(topological solution)} \\
%\textrm{either}\ \left({u}_{1,n},{u}_{2,n}\right) \to (-\infty,0)  \ \ \textrm{a.e. or}\ \ \left({u}_{1,n},{u}_{2,n}\right) \to (0,-\infty) \quad\mbox{a.e. on }~ \mathbb{T} \quad\mbox{as }~ {n}\to\infty,
% \quad&\mbox{(mixed type solution)} \\
 % \left({u}_{1,n},{u}_{2,n}\right) \to (-\infty,-\infty) \quad\mbox{a.e. on }~ \mathbb{T} \quad\mbox{as }~ {n}\to\infty,
 %\quad&\mbox{(nontopological solution)}
%\end{aligned}
%\end{equation}
In order to describe our main result precisely, we introduce some notations.
Firstly,  let  $G(x,y)$ be the Green's function  satisfying
\begin{equation*}
-\Delta_x G(x,y)=\delta_y-\frac{1}{|\mathbb{T}|},\quad\int_{\mathbb{T}}G(x,y)dy=0,
\end{equation*}
where $|\mathbb{T}|$ is the measure of $\mathbb{T}$ (see \cite{Au}).  We denote the regular part of $G(x,y)$ by \[\gamma(x,y)=G(x,y)+\frac{1}{2\pi}\ln|x-y|.\]
Let
\begin{equation}\label{u_0}
u_0(x)=-8\pi \sum_{i=1}^N m_{i}G(x,p_i).
\end{equation}
Now we have  the following Brezis-Merle type alternatives  result for \eqref{main_eq}.
\begin{theorem}\label{BrezisMerletypealternatives}
Let $Z\equiv\cup_{i=1}^N\{p_{i}\}$.
We assume that  $\{({u}_{1,n}, {u}_{2,n})\}$ is a sequence of solutions of \eqref{main_eq}.

(i)  as ${ n \to\infty}$, up to subsequences, ${u}_{1,n}$ satisfies one of the followings:

\begin{description}
\item[(f1)] ${u}_{1,n}\to 0$ uniformly on any compact subset of $\mathbb{T}\setminus  Z$, or
\item[(f2)]   $\ulm-2\ln\varepsilon_n-u_0\to\hat{w}$ in $C^{1}_{\textrm{loc}}(\mathbb{T})$, where $\hat{w}$ satisfies  $\Delta \hat{w} +e^{\hat{w}+u_0}=8\pi\mathfrak{M}$, or
\item[(f3)]  there exists a nonempty finite set $\mathfrak{B}=\{{q}_1,\cdots,{q}_k\}\subset\mathbb{T}$ and $k$-number of sequences  of points $q^{j}_{ n }\in\mathbb{T}$ such that $\lim_{ n \to\infty }q^{j}_{ n }={q}_j$, $\left({u}_{1,n}-2\ln\varepsilon_n\right)(q^{j}_{ n })\to+\infty$, and ${u}_{1,n}-2\ln\varepsilon_n\to-\infty$ uniformly on any compact subset of $\mathbb{T}\setminus B$. Moreover,
$\frac{1}{\varepsilon_n^2}e^{{u}_{1,n}}\left(1-e^{{u}_{1,n}}\right)\to\sum_j\alpha_j\delta_{{q}_j},\ \ \alpha_j\ge8\pi,$
in the sense of measure.
\end{description}

(ii)  as ${ n \to\infty}$, up to subsequences, ${u}_{2,n}$ satisfies one of the followings:

\begin{description}
\item[(s1)]  there is a constant $c_0>0$, independent of $n\ge1$, satisfying  $\|{u}_{2,n}\|_{L^\infty(\mathbb{T})}\le c_0 \varepsilon_n^2$, or
\item[(s2)]   $\sup_{\mathbb{T}}\left({u}_{2,n}-2\ln\varepsilon_n \right)\to-\infty$.
\end{description}

\end{theorem}
One might expect that the equation \eqref{main_eq}  is almost decoupled system, and can be regarded  as a perturbation of the following equation arising from the classical Chern-Simons-Higgs (CSH) model:
\begin{equation}\label{cs_eq}
 \Delta u_{\varepsilon}=\frac{1}{\varepsilon^2}e^{u_\varepsilon}\left( e^{u_\varepsilon}-1\right)+8\pi\sum_{i=1}^N m_{i}\delta_{p_i}
\quad \mbox{ in }\mathbb{T}.
\end{equation}
Based on the arguments for Brezis-Merle type alternatives, the   corresponding result  for (CSH) equation \eqref{cs_eq} has been proved in \cite{CK}   (see also  \cite{BCCT,BT,BM,CK, NT1,NT2}). However,  we note that unlike  the first component $u_{1,n}$,  the asymptotic behavior $\bf{(s2)}$ for the second component $u_{2,n}$ in Theorem \ref{BrezisMerletypealternatives} does not follow the usual behavior for (CSH) model.  Moreover, we also remark that the refinement for  the asymptotic behavior of $u_{2,n}$ is essential   to obtain the concentration property for  blow up component, that is, $\bf{(f3)}$.

The paper is organized as follows. In Section 2, we review some preliminaries.
In Section 3, we analyze the asymptotic behavior of solutions and prove Theorem   \ref{BrezisMerletypealternatives}.

\section{Preliminaries}\label{sec2}
Let $w$ satisfy
\begin{equation}
\begin{aligned} \label{limitingpro}
\Delta w + e^w(1-e^w)=4\pi m\delta_0\ \ \textrm{in}\ \mathbb{R}^2.
\end{aligned}
\end{equation}
We recall the property of  a solution $w$ to \eqref{limitingpro}  as follows.
\begin{lemma}\label{lemma2.1}
\cite{BM,ChL2} \cite[Lemma 3.2]{CK} Let $m$ be a nonnegative integer, and $w$ be a  solution of \eqref{limitingpro}.\\
If $e^w(1-e^w) \in L^1(\mathbb{R}^2)$, then either

(i) $w(x)\to0$ as $|x|\to\infty$, or

(ii) $w(x)=-\beta\ln |x| + O(1)$ near $\infty$, where
$\beta=-2m+{\frac{1}{2\pi}}\int_{\mathbb{R}^2} e^w(1-e^w)dx.$

\noindent
Assume that $w$ satisfies the boundary condition (ii). Then we have
\[\int_{\mathbb{R}^2}e^{2w} dx = \pi (\beta^2-4\beta-4m^2-8m), \  \textrm{and}\
\int_{\mathbb{R}^2}e^{w} dx = \pi (\beta^2-2\beta-4m^2-4m).\]
In particular,
$\int_{\mathbb{R}^2} e^w(1 -e^w)dx >8\pi(1 +m).$
\end{lemma}
When $m=0$, it has been known that $\int_{\mathbb{R}^2}e^w(1-e^w)dx$ depends on the maximum value of $w$, and has a lower bound as follows.
\begin{lemma}\cite[Theorem 2.1]{CFL} \cite[Theorem 3.2]{CHMY}  \cite[Theorem 2.2]{SYa1}\label{propertyofentiresolution}
Let $m=0$, and $w$ be a  solution of \eqref{limitingpro} with $e^w(1-e^w)\in L^1(\mathbb{R}^2)$. Then, $w(x)$ is smooth,  radially symmetric with respect to some point $x_0$ in $\mathbb{R}^2$, and strictly decreasing function of $r=|x-x_0|$.

Assume $w(r;s)$ be the radially symmetric  solution with respect to $0$ of \eqref{limitingpro} such that \[\lim_{r\to0}w(r; s)=s,\ \  \textrm{and}\ \  \lim_{r\to0}w'(r; s)=0,\]
where $w'$ denotes $\frac{dw}{dr}(r; s)$, and let us set
\begin{equation}\label{defbetas}
\beta(s)\equiv\frac{1}{2\pi}\int_{\mathbb{R}^2} e^{w(r; s)}(1-e^{w(r; s)})dx=\int^\infty_0 e^{w(r; s)}(1-e^{w(r; s)})rdr.
\end{equation}Then one has

(i) $\beta(0)=0$ and $w(\cdot;0)\equiv0$;

(ii) $\beta: (-\infty, 0)\rightarrow(4,+\infty)$ is strictly increasing, bijective, and
$$\lim_{s\to-\infty}\beta(s)=4, \textrm{ and }\lim_{s\to 0_-}\beta(s)=+\infty.$$
\end{lemma}
The following a priori estimates of the solutions to \eqref{main_eq} has been known.
\begin{lemma} \cite[Proposition 5.1]{HLTY} \label{lemma_RT} Let $({u}_{1,n},{u}_{2,n})$ be solutions of \eqref{main_eq} over $\mathbb{T}$. Then \[{u}_{1,n}(x)<0, \quad {u}_{2,n}(x)<0\ \ \ \ \textrm{for any}\ \  \ x\in\mathbb{T}.\]
\end{lemma}
By the assumptions \eqref{assume1}, \eqref{assume2}, \eqref{assume3},  and the definitions of $\varepsilon_n,\ \sigma_n>0$ in \eqref{epsilon}, we note that
\begin{equation}\label{assume4}0<(\beta_n-\alpha_n)(\beta_n+\alpha_n)=\frac{\sigma_n}{\varepsilon_n^2}\le\mathfrak{N}.\end{equation}
Then we have the uniform $L^1(\mathbb{T})$ boundedness of  the main nonlinear terms in \eqref{main_eq}  with respect to $ n >0$.
\begin{lemma}\label{cor1}Let $({u}_{1,n},{u}_{2,n})$ satisfy  \eqref{main_eq} over $\mathbb{T}$. Then we have
\begin{equation}\label{int1} \int_{\mathbb{T}}\frac{1}{\varepsilon_n^2}e^{{u}_{1,n}}\left(1-e^{{u}_{1,n}}\right)dx=\int_{\mathbb{T}}\frac{1}{\varepsilon_n^2}e^{{u}_{1,n}}\left|1-e^{{u}_{1,n}}\right|dx\le 8\pi \mathfrak{M}, \quad \textrm{where}\ \ \mathfrak{M}=\sum_{i=1}^N m_{i},\end{equation}
and \begin{equation}\label{int2} \int_{\mathbb{T}}\frac{1}{\varepsilon_n^2}e^{{u}_{2,n}}\left(1-e^{{u}_{2,n}}\right)dx=\int_{\mathbb{T}}\frac{1}{\varepsilon_n^2}e^{{u}_{2,n}}\left|1-e^{{u}_{2,n}}\right|dx\le 2\mathfrak{N}|\mathbb{T}|.\end{equation}\end{lemma}
\begin{proof}By integrating \eqref{main_eq} over $\mathbb{T}$, we can obtain  \begin{equation}\label{main_eq_int}
\left\{\begin{array}{l}
\int_{\mathbb{T}}\frac{1}{\varepsilon_n^2} \left\{e^{{u}_{1,n}}(1-e^{{u}_{1,n}})+\sigma_n^2 e^{{u}_{2,n}}(1-e^{{u}_{1,n}})\right\}dx= \int_{\mathbb{T}}\frac{\sigma_n}{\varepsilon_n^2} (e^{{u}_{1,n}}+e^{{u}_{2,n}})(1-e^{{u}_{2,n}})dx +8\pi\mathfrak{M},\\ \\
\int_{\mathbb{T}}\frac{1}{\varepsilon_n^2}\left\{e^{{u}_{2,n}}(1-e^{{u}_{2,n}})+\sigma_n^2 e^{{u}_{1,n}}(1-e^{{u}_{2,n}})\right\}dx=\int_{\mathbb{T}}\frac{\sigma_n}{\varepsilon_n^2} (e^{{u}_{1,n}}+e^{{u}_{2,n}})(1-e^{{u}_{1,n}})dx.
\end{array}
\right.
\end{equation}
By using Lemma \ref{lemma_RT} and the assumption \eqref{assume4}, we get the estimations \eqref{int1} and \eqref{int2}.
\end{proof}

 We recall the following Harnack inequality type result.
\begin{lemma}(\cite{BT,GT})\label{harnarkineq}
Let $D\subseteq\mathbb{R}^2$ be a smooth bounded domain and $v$ satisfy:
$$-\Delta v=g\ \textrm{in}\ D,$$
with $g\in L^p(D)$, $p>1$. For any subdomain $D'\subset\subset D$, there exist two positive constants $\sigma\in(0,1)$ and $\tau>0$,
depending on $D'$ only such that:

$$(i)\ \textrm{if}\ \sup_{\partial D} v\le C,\ \textrm{then}\ \sup_{D'} v\le\sigma\inf_{D'}v+(1+\sigma)\tau\|g\|_{L^p}+(1-\sigma)C,$$

$$(ii)\ \textrm{if}\ \inf_{\partial D} v\ge -C,\ \textrm{then}\ \sigma\sup_{D'} v\le\inf_{D'}v+(1+\sigma)\tau\|g\|_{L^p}+(1-\sigma)C.$$
\end{lemma}
The linear operator $L_n$ is defined by
\begin{equation}\label{eq10}
L_n(S):=\Delta S-\frac{1}{\varepsilon_n^2} S.
\end{equation}In order to improve the asymptotic behavior of $u_{2,n}$ for \eqref{eq0}, the following result will be useful.
\begin{theorem}\cite[Theorem 2.4 ]{FLL}\label{theoremb}
The operator
\begin{equation*}
L_{n}: W^{2,2}(\mathbb{T})\to L^2(\mathbb{T})
\end{equation*}
is an isomorphism. Moreover, for any $S\in W^{2,2}(\mathbb{T})$ and $g\in L^2(\mathbb{T})$ satisfying $L_{n}(S)=g$, there exists a positive constant $C>0$, independent of ${\varepsilon_n}>0$, such that
\begin{equation}\label{eq19} \left\{\begin{array}{l}
 \|S\|_{L^\infty(\mathbb{T})} \leq C  {\varepsilon_n}\|g\|_{L^2(\mathbb{T})},
\\ \|S\|_{L^\infty(\mathbb{T})}\leq C  {\varepsilon_n^2} \|g\|_{L^\infty(\mathbb{T})}\ \ \textrm{if}\ \ g\in L^\infty(\mathbb{T}).
\end{array}
\right.
\end{equation}
\end{theorem}
We set $u_{1,n}=v_{1,n}+u_0$, and assume $|\mathbb{T}|=1$. Then (\ref{main_eq}) is equivalent to
\begin{equation}\label{eq2}
\left\{\begin{array}{l}
\Delta {v}_{1,n}=\frac{1}{\varepsilon_n^2} \left\{e^{{v}_{1,n}+u_0}(e^{{v}_{1,n}+u_0}-1)+\sigma_n^2 e^{{u}_{2,n}}(e^{{v}_{1,n}+u_0}-1)-\sigma_n (e^{{v}_{1,n}+u_0}+e^{{u}_{2,n}})(e^{{u}_{2,n}}-1)\right\}\\\quad\quad\ +8\pi\mathfrak{M},\\ \\
\Delta {u}_{2,n}=\frac{1}{\varepsilon_n^2}\left\{e^{{u}_{2,n}}(e^{{u}_{2,n}}-1)+\sigma_n^2 e^{{v}_{1,n}+u_0}(e^{{u}_{2,n}}-1)-\sigma_n (e^{{v}_{1,n}+u_0}+e^{{u}_{2,n}})(e^{{v}_{1,n}+u_0}-1)\right\},
\end{array}
\right.
\end{equation}

\begin{lemma}\label{grad_sol} Let $({u}_{1,n},{u}_{2,n})$ be a sequence of solutions for  \eqref{main_eq} over $\mathbb{T}$. Then there exists a constant $C>0$, independent of ${n}\ge1$, such that
\[\left\|\nabla v_{1,n} \right\|_{L^\infty(\mathbb{T})}+\left\|\nabla {u}_{2,n} \right\|_{L^\infty(\mathbb{T})}\le \frac{C}{\varepsilon_n}.\]
\end{lemma}

\begin{proof}By  the Green's representation formula, we have that
\[v_{1,n}(x)-\int_\mathbb{T} v_{1,n}dy=\int_\mathbb{T} \frac{1}{\varepsilon_n^2} \left\{e^{{u}_{1,n}}(1-e^{{u}_{1,n}})+\sigma_n^2 e^{{u}_{2,n}}(1-e^{{u}_{1,n}})-\sigma_n (e^{{u}_{1,n}}+e^{{u}_{2,n}})(1-e^{{u}_{2,n}})\right\}G(x,y)dy,
\]and
\[u_{2,n}(x)-\int_\mathbb{T} u_{2,n}dy=\int_\mathbb{T} \frac{1}{\varepsilon_n^2} \left\{e^{{u}_{2,n}}(1-e^{{u}_{2,n}})+\sigma_n^2 e^{{u}_{1,n}}(1-e^{{u}_{2,n}})-\sigma_n (e^{{u}_{1,n}}+e^{{u}_{2,n}})(1-e^{{u}_{1,n}})\right\}G(x,y)dy.
\]
In view of  Lemma \ref{lemma_RT},  Lemma \ref{cor1}, and \eqref{assume4}, we see that there are constants $c_0, c_1, C>0$, independent of $n\ge1$, satisfying
\begin{equation}\begin{aligned}
&\left|\nabla_x {v}_{1,n} (x)\right|+\left|\nabla_x {u}_{2,n} (x)\right|
\\&\le\sum_{i=1}^2\left|\int_{B_d(x)}\frac{1}{\varepsilon_n^2}e^{{u}_{i,n}(y)}\left(1-e^{{u}_{i,n}}\right)\left(-\frac{x-y}{2\pi|x-y|^2}+\nabla \gamma(x,y)\right)dy\right|+c_0+4\mathfrak{N}+2\mathfrak{N}^2\varepsilon_n^2\\
&\leq \sum_{i=1}^2\left[\frac{   \left\|e^{{u}_{i,n}}\left(1-e^{{u}_{i,n}}\right)\right\|_{L^\infty(\mathbb{T})}}{2\pi \varepsilon_n^2}\left(\int_{|x-y|\leq\varepsilon_n } \frac{1}{|x-y|}dy\right)
+\int_{\varepsilon_n \le|x-y|\leq d}\frac{ e^{{u}_{i,n}(y)}\left|1-e^{{u}_{i,n}}\right|}{2\pi\varepsilon_n^2|x-y|}dy\right] +c_1
 \le  \frac{C}{\varepsilon_n}.
\end{aligned}\end{equation}

\end{proof}
\section{The asymptotic behavior of solutions}

%-------------------------------------------------------------------------

We first obtain the following result by   applying the arguments in  \cite[Lemma 4.1]{CK}.
\begin{lemma}\label{uniformestimates}
Let  $({u}_{1,n}, {u}_{2,n})$ be
  a sequence of solutions for \eqref{main_eq}.

(i) if $\lim_{ n \to\infty }\Big(\inf_\mathbb{T}|{u}_{1,n}|\Big)=0,$ then we have \begin{equation}\label{K_1}\lim_{{ n \to\infty}}\|{u}_{1,n}\|_{L^\infty(K)}=0\ \ \ \textrm{for any compact set }\ \ K\subset\mathbb{T}\setminus Z.\end{equation}

(ii) if $\lim_{ n \to\infty }\Big(\inf_\mathbb{T}|{u}_{2,n}|\Big)=0,$ then there is a constant $c_0>0$, independent of $n\ge1$, satisfying  \begin{equation}\label{K_2} \|{u}_{2,n}\|_{L^\infty(\mathbb{T})}\le c_0 \varepsilon_n^2.\end{equation}
\end{lemma}
\begin{proof}
 (i) We assume that $\lim_{ n \to\infty }\Big(\inf_\mathbb{T}|{u}_{1,n}|\Big)=0$ holds true.  Suppose that
\begin{equation}|{u}_{1,n}(x_{ n })|=\inf_\mathbb{T}|{u}_{1,n}|, \ \ \textrm{and} \ \lim_{{ n \to\infty}}{u}_{1,n}(x_{ n })=0.\label{to0}\end{equation}
Passing to a subsequence (still denoted by ${u}_{1,n}$),
we   assume that $\lim_{ n \to\infty }x_{ n }=x_0\in \mathbb{T}$. Now we consider the following two cases depending on the location of $x_0$.

\medskip
\noindent
\textbf{Case 1.} $x_0\notin Z$.\\
Fix a small constant  $d>0$ satisfying  $B_d(x_0)\cap Z=\emptyset$.
We argue by contradiction and suppose that there exist a compact set $K\subset\mathbb{T}\setminus Z$, a constant $c_K>0$, and a sequence $\{z_{ n }\}\subset K$  such that
$\sup_K|{u}_{1,n}|=|{u}_{1,n}(z_{ n })|\ge c_K>0$ for  large $\lm0$. Let $K_1\subset\mathbb{T}\setminus Z$ be  a connected compact set  satisfying $B_d(x_0)\cup K \subset K_1$.
  By  ${u}_{1,n}(z_{ n })\le-c_K<0$ and  Lemma \ref{propertyofentiresolution}, we have a constant  $s_1<0$ satisfying
\begin{equation}\label{chs1}\beta(s_1)  >16\mathfrak{M}\ \textrm{and}\ -c_K<s_1<0.\end{equation}
In view of the intermediate value theorem, there is a sequence  $y_{ n }\in K_1$ such that ${u}_{1,n}(y_{ n })=s_1$.\\
For $k=1,2$, let $ \bar{u}_{k, n }(x)=\left({u}_{k,n}  \right)({\varepsilon_n} x+y_{ n })$, where $x\in\mathbb{T}_{y_{ n }}\equiv\{\ x\in\mathbb{R}^2\ |\ {\varepsilon_n} x+y_{ n }\in K_1\ \}$.
Then $ \bar{u}_{1, n }$ satisfies
\begin{equation*}
\begin{aligned}
\left\{
 \begin{array}{ll}
 \Delta  \bar{u}_{1, n }+ e^{\bar{u}_{1, n }(x)}(1- e^{\bar{u}_{1, n }(x)})+\sigma_n^2 e^{\bar{u}_{2,n}}(1- e^{\bar{u}_{1, n }(x)})-\sigma_n (e^{\bar{u}_{1,n}}+e^{\bar{u}_{2,n}})(1-e^{\bar{u}_{2,n}})=0
\    \textrm{in}\ \ \mathbb{T}_{y_{ n }},
 \\  \bar{u}_{1, n }(0)=s_1,
 \\ \int_{\mathbb{T}_{y_{ n }}} \left|e^{ \bar{u}_{1, n }(x) }\left(1-e^{\bar{u}_{1,n}(x)}\right)\right| dx\le {8\pi\mathfrak{M}}.
 \end{array}\right.
\end{aligned}
\end{equation*}
From  Lemma \ref{grad_sol} and $W^{2,p}$ estimation, we see that $ \bar{u}_{1, n }$ is bounded in $C^{1,\sigma}_{\textrm{loc}}(\mathbb{T}_{y_{ n }})$ for some $\alpha\in(0,1)$.  Then we see that  passing to a subsequence,  $ \bar{u}_{1, n }$ converges in
$C^1_{\textrm{loc}}(\mathbb{R}^2)$ to a function $u_*$, which is a solution of
\begin{equation}
\begin{aligned}
\left\{
 \begin{array}{ll}
 \Delta u_*+ e^{u_*}(1-e^{u_*}) =0\ \textrm{in}\ \mathbb{R}^2,
 \\ u_*(0)=s_1,
 \\ \int_{\mathbb{R}^2}|e^{u_*}(1-e^{u_*})|dx\le {8\pi\mathfrak{M}}.
 \end{array}\right.
\end{aligned}
\end{equation}
By    Lemma \ref{propertyofentiresolution}, we see that $u_*$ is
radially symmetric with respect to a point $\bar{p}$ in $\mathbb{R}^2$. In view of   Lemma \ref{propertyofentiresolution} and \eqref{chs1}, we get
\begin{equation}
\begin{aligned}
{8\pi\mathfrak{M}}
\ge \Big|\int_{\mathbb{R}^2} e^{u_*}(1-e^{u_*}) dx\Big|
=2\pi|\beta(u_*(\bar{p}))|\ge2\pi|\beta(s_1)|
>{32\pi\mathfrak{M}},
\end{aligned}
\end{equation}
which implies a contradiction, and we prove that    \eqref{K_1} holds true for  Case 1.

\medskip
\noindent
\textbf{Case 2.} $x_0=p_i\in Z$ for some $i\in\{1,\cdots,N\}$.
\\ Choose a   constant $r_0>0$ such that $B_{r_0}(x_0)\cap Z=\{x_0\}$.
For the  simplicity, we assume that $x_0=0$.   We claim that
\begin{equation}\lim_{{ n \to\infty}}\left(\inf_{|x|=r_0}|\ulm(x)|\right)=0.\label{cmlt}\end{equation}
Once we obtain the claim \eqref{cmlt}, the argument in Case 1 implies \eqref{K_1}. To prove   \eqref{cmlt}, we argue by contradiction,  and suppose that  there is a constant $\tau_0>0$ satisfying, up to a subsequence, $\lim_{{ n \to\infty}}\left(\inf_{|x|=r_0}|{u}_{1,n}(x)|\right)\ge\tau_0$. By  Lemma \ref{lemma_RT}, we have ${u}_{1,n}<0$, and thus
\begin{equation}\label{gam}\lim_{{ n \to\infty}}\left(\sup_{|x|=r_0}\ulm(x)\right)<-\tau_0.\end{equation} We consider  the following two cases:

\medskip

\noindent
(1)  $\lim_{{ n \to\infty}}\left( \frac{|x_{n}|}{\varepsilon_n}\right)<+\infty$.
\\ Let  $v_{ n }(x)=\ulm(x)-4m_i\ln|x|$ near $x=0$. Then  $v_{ n }$ is  smooth  in $B_d(0)$. We apply a scaling for $v_{ n }$ as follows:
\[
\hat{v}_{ n }(x)= v_{ n }(|x_{ n }|x)+4m_i\ln|x_{ n }| \ \ \textrm{for} \ \ |x|\le\frac{r_0}{|x_{ n }|}.
\]
Then $\hat{v}_{ n }$ satisfies
\begin{equation}\label{vmain_eq}
\left\{\begin{array}{l}
\Delta \hat{v}_{ n }=\frac{|x_{n}|^2}{\varepsilon_n^2} \Bigg[|x|^{4m_i}e^{\hat{v}_{ n }(x)}(|x|^{4m_i}e^{\hat{v}_{ n }(x)}-1)+\sigma_n^2 e^{{u}_{2,n}(|x_n|x)}(|x|^{4m_i}e^{\hat{v}_{ n }(x)}-1)
\\ \quad\quad\quad -\sigma_n (|x|^{4m_i}e^{\hat{v}_{ n }(x)}+e^{{u}_{2,n}(|x_n|x)})(e^{{u}_{2,n}(|x_n|x)}-1)\Bigg]\ \ \textrm{in}\  B_{\frac{r_0}{|x_{ n }|}}(0),\\
\int_{B_{\frac{r_0}{|x_{ n }|}}(0)}\frac{|x_{n}|^2}{\varepsilon_n^2}|x|^{4m_i} e^{\hat{v}_{ n }(x)   }\left(1-|x|^{4m_i} e^{\hat{v}_{ n }(x)   }\right)dx\le 8\pi\mathfrak{M}
\end{array}
\right.
\end{equation}
By \eqref{to0}, we have
\begin{equation}\label{limv}\lim_{{ n \to\infty}}\hat{v}_{ n }
\left(\frac{x_{ n }}{|x_{ n }|}\right)
=\lim_{{ n \to\infty}}\left(\ulm(x_{ n })
  \right)=0.\end{equation}  In view of  Lemma \ref{grad_sol} and the assumption $\lim_{{ n \to\infty}}\left(  \frac{|x_{ n }|}{\varepsilon_n}\right)<+\infty$, we see that $\hat{v}_{ n }$ is bounded in $C^0_{\textrm{loc}}\left(B_{r_1}(0)\right)$.  Thus, we have a point $y_0\in\mathbb{S}^1$,  a constant  $c_0\ge0$, and a function $\hat{v}$ satisfying, passing to a subsequence, \[\lim_{{ n \to\infty}}\frac{x_{ n }}{|x_{ n }|}=y_0\in \mathbb{S}^1,\ \lim_{{ n \to\infty}}\left(\frac{|x_{n}|}{\varepsilon_n}\right)=c_0\ge0, \ \textrm{and}\ \hat{v}_{ n }\to\hat{v}\ \textrm{ in}\ C_{\textrm{loc}}^1\left(B_{r_1}(0)\right).\]  We note that $\hat{u}(x)=\hat{v}(x)+4m_i\ln|x|$ satisfies
\[
\Delta\hat{u}+c_0^2e^{\hat{u}}(1-e^{\hat{u}})=8\pi m_i\delta_0\ \ \ \textrm{in}\ \ \RN.
\]
From Lemma \ref{lemma_RT}, we have $\hat{u}\le 0$ in $\RN$. Since  $\hat{u}(y_0)= 0$,  we have $\hat{u}\equiv0$ by Hopf Lemma, which  contradicts.

\medskip
\noindent
(2)
$\lim_{{ n \to\infty}}\left( \frac{|x_{n}|}{\varepsilon_n}\right)= +\infty$.\\
Lemma \ref{propertyofentiresolution} implies there is a constant  $s_2<0$ such that
\[
\beta(s_2)  >16\mathfrak{M}\ \textrm{and}\ -\tau_0<s_2<0,
\]
where $\tau_0$ is the constant in \eqref{gam}.
We can also choose $\hat{y}_{ n }$ on the line segment  joining $x_{ n }$ and $\frac{r_0 x_{ n }}{|x_{ n }|}$ such that ${u}_{1,n}(\hat{y}_{ n })=s_2$ and $|\hat{y}_{ n }|\ge |x_{ n }|$ by the intermediate value theorem.\\
For $k=1,2$, let $\hat{u}_{k, n }(x)=\left({u}_{k,n}  \right)({\varepsilon_n} x+\hat{y}_{ n })$ for $x\in\hat{\mathbb{T}}_{\hat{y}_{ n }}\equiv\{\ x\in\mathbb{R}^2\ |\ {\varepsilon_n} x+\hat{y}_{ n }\in B_{\frac{|x_{ n }|}{2}}(\hat{y}_{ n }) \ \}$.  Here we note that $0\notin B_{\frac{|x_{ n }|}{2}}(\hat{y}_{ n }) $.
Then $\hat{u}_{1, n }$ satisfies
\begin{equation}
\begin{aligned}
\left\{
 \begin{array}{ll}
 \Delta \hat{u}_{1, n }
+ e^{\hat{u}_{1, n }(x)}(1- e^{\hat{u}_{1, n }(x)})+\sigma_n^2 e^{\hat{u}_{2,n}}(1- e^{\hat{u}_{1, n }(x)})-\sigma_n (e^{\hat{u}_{1,n}}+e^{\hat{u}_{2,n}})(1-e^{\hat{u}_{2,n}})=0\ \ \  \textrm{in}\ \ \hat{\mathbb{T}}_{\hat{y}_{ n }},
 \\ \hat{u}_{1, n }(0)=s_2,
 \\ \int_{\hat{\mathbb{T}}_{\hat{y}_{ n }}} |e^{\hat{u}_{1, n }}(1-e^{\hat{u}_{1, n }})| dx\le {8\pi\mathfrak{M}}.
 \end{array}\right.\label{nonl}
\end{aligned}
\end{equation} Using the same argument as in Case 1, we get a  contradiction by comparing  $L^1$ norm of $ e^{\hat{u}_{1, n }(x)}(1- e^{\hat{u}_{1, n }(x)})$.  Thus the claim \eqref{cmlt} holds true.  Then we can again apply the arguments in  Case 1  and prove \eqref{K_1} holds true.

(ii) By applying the similar argument  in (i) to ${u}_{2,n}$, we can also prove that  if $\lim_{ n \to\infty }\Big(\inf_\mathbb{T}|{u}_{2,n}|\Big)=0,$ then   $\lim_{{ n \to\infty}}\|{u}_{2,n}\|_{L^\infty(\mathbb{T})}=0$. We note that ${u}_{2,n}$ satisfies
\begin{equation*}
\Delta {u}_{2,n} -\frac{1}{\varepsilon_n^2}{u}_{2,n}=\frac{1}{\varepsilon_n^2}\left\{e^{{u}_{2,n}}(e^{{u}_{2,n}}-1)- {u}_{2,n}+\sigma_n^2 e^{{u}_{1,n}}(e^{{u}_{2,n}}-1)-\sigma_n (e^{{u}_{1,n}}+e^{{u}_{2,n}})(e^{{u}_{1,n}}-1)\right\}.
\end{equation*}
 By  applying  Taylor's theorem to the function $e^{{u}_{2,n}}(e^{{u}_{2,n}}-1)$ and using  Theorem  \ref{theoremb} , we see that there are constants $c_1, c_2>0$, independent of $n\ge1$, satisfying
 \begin{equation*}
   \begin{aligned}
 \|{u}_{2,n}\|_{L^\infty(\mathbb{T})}& \leq c_1
 \| e^{{u}_{2,n}}(e^{{u}_{2,n}}-1)- {u}_{2,n}+\sigma_n^2 e^{{u}_{1,n}}(e^{{u}_{2,n}}-1)-\sigma_n (e^{{u}_{1,n}}+e^{{u}_{2,n}})(e^{{u}_{1,n}}-1) \|_{L^\infty(\mathbb{T})}
\\& \le c_2\left( \|{u}_{2,n}\|_{L^\infty(\mathbb{T})}^2+\sigma_n^2  + 2\sigma_n    \right).\end{aligned}\end{equation*}
By $\lim_{{ n \to\infty}}\|{u}_{2,n}\|_{L^\infty(\mathbb{T})}=0$ and \eqref{assume4}, we can conclude that \eqref{K_2} holds.
\end{proof}

As a corollary of Lemma \ref{uniformestimates}, we get the following proposition.
\begin{proposition}\label{alternatives}
Let $({u}_{1,n}, \nl)$ be a sequence of solutions of \eqref{main_eq}.

(1) up to subsequences,  $u_{1,n}$ satisfies one of the followings:\\
(1a) ${u}_{1,n}\to0$ uniformly on any compact subset of $\mathbb{T}\setminus Z$ as ${ n \to\infty}$, or\\
(1b) there exists a constant $\nu_0>0$ such that $\sup_{ n \to\infty }\Big(\sup_\mathbb{T} {u}_{1,n}\Big)\le-\nu_0$.

(2) up to subsequences,  $u_{2,n}$ satisfies one of the followings:
\\
(2a)  there is a constant $c_0>0$, independent of $n\ge1$, satisfying  $\|{u}_{2,n}\|_{L^\infty(\mathbb{T})}\le c_0 \varepsilon_n^2$, or
\\
(2b) there exists a constant $\nu_0>0$ such that $\sup_{ n \to\infty }\Big(\sup_\mathbb{T} {u}_{2,n}\Big)\le-\nu_0$.

\end{proposition}

For $k=1,2$, let us denote
\[
w_{k, n }= {u}_{k,n}   -2\ln\varepsilon_n \quad\mbox{in }~ \mathbb{T}.
\]
For the second component, we have the following result.
\begin{lemma}\label{2ndalternatives}
Let $({u}_{1,n}, \nl)$ be a sequence of solutions of \eqref{main_eq}. \\
If   $u_{2,n}$ satisfies   Proposition \ref{alternatives}-(2b), then
 $\sup_{\mathbb{T}}\left({w}_{2,n}\right)\to-\infty$ as $n\to\infty$.
\end{lemma}
\begin{proof}
The function $w_{2, n }$ satisfies
 \begin{equation}\label{w2main_eq}\begin{aligned}
 \Delta {w}_{2,n}+ e^{{w}_{2,n}}\left\{1-\varepsilon_n^2 e^{{w}_{2,n}} - \sigma_n (1-e^{{u}_{1,n}})\right\}+\frac{\sigma_n^2}{\varepsilon_n^2} e^{{u}_{1,n}}(1-\varepsilon_n^2 e^{{w}_{2,n}})
-\frac{\sigma_n}{\varepsilon_n^2}e^{{u}_{1,n}}(1-e^{{u}_{1,n}})
 =0\ \mbox{in }\  \mathbb{T}.
\end{aligned}
\end{equation}
 By using Lemma \ref{cor1} and  $\frac{\sigma_n}{\varepsilon_n^2}\le\mathfrak{N}$, we see that
\begin{equation}\begin{aligned}\label{contminimal}
\sigma_n(8\pi\mathfrak{M})&\ge \int_{\mathbb{T}}\frac{\sigma_n}{\varepsilon_n^2}e^{{u}_{1,n}}(1-e^{{u}_{1,n}})dx
\\&= \int_{\mathbb{T}} e^{{w}_{2,n}}\left\{1-\varepsilon_n^2 e^{{w}_{2,n}} - \sigma_n (1-e^{{u}_{1,n}})\right\}+\frac{\sigma_n^2}{\varepsilon_n^2} e^{{u}_{1,n}}(1-\varepsilon_n^2 e^{{w}_{2,n}})
dx
\\&\ge  \frac{(1-e^{-\nu_0})}{2}\int_{\mathbb{T}}   e^{{w}_{2,n}}dx,
\end{aligned} \end{equation}
 which implies \begin{equation}\label{w2main_ieq}\lim_{n\to\infty}\int_{\mathbb{T}}   e^{{w}_{2,n}}dx=0.\end{equation}
We consider the following two cases:

\medskip\noindent
\textit{(case 1)} $\sup_{\mathbb{T}}\wlmt\le C$ for some constant $C>0$.  \\
 In this case, by applying  Lemma \ref{harnarkineq} to  \eqref{w2main_eq},  we see that  either  $\wlmt$ is uniformly bounded in $L^\infty(\mathbb{T})$, or  $\sup_{\mathbb{T}}\left({w}_{2,n}\right)\to-\infty$  as $n\to\infty$. By  \eqref{w2main_ieq}, we see that \textit{(case 1)}  implies  $\sup_{\mathbb{T}}\left({w}_{2,n}\right)\to-\infty$ as $n\to\infty$.

\medskip\noindent
\textit{(case 2)} $\lim_{{ n \to\infty}}\sup_{\mathbb{T}}\wlmt=+\infty.$

Following \cite{BM}, we say that a point $q\in \mathbb{T}$ is a blow-up point for $\{w_{2, n }\}$
if there exists a sequence $\{x_{ n ,q}\}$ such that
\begin{equation}\label{blowuppoint}x_{ n ,q} \to q \quad\mbox{and}\quad w_{2, n } (x_{ n ,q}) \to \infty
 \quad\mbox{as }~  n \to\infty. \end{equation}
Let $S_2\subset\mathbb{T} $ be   the set of  blow-up points for $\{w_{2, n }\}$.
We claim that $S_2=\emptyset$ so that  \textit{(case 2)} cannot occur. In order to prove this claim, we argue by contradiction, and suppose that  there is a point $p\in S_2$.
We are going to show  the following ``minimal mass" result:
\begin{equation}\label{massofsingularpoint2} \liminf_{ n \to\infty }\int_{B_d(p)} \left| e^{{w}_{2,n}}\left(1-\varepsilon_n^2 e^{{w}_{2,n}} \right)\right| dx \ge 8\pi\ \ \textrm{for any}\ \ d>0.\end{equation}  In order to obtain \eqref{massofsingularpoint2}, we will apply the arguments in \cite[Lemma 4.2]{CK}.  We first choose a small constant $d>0$ and   a sequence of points
$\{x_{ n }\}$ in $B_d(p)$ such that
$w_{2, n }(x_{ n }) =\sup_{x\in B_{d}(p)} w_{2, n }(x)\to\infty$ and $x_{ n }\to p$ as ${n}\to\infty$.
Let
\[ s_{ n } = \exp\Big(-\frac{1}{2}w_{2, n }(x_{ n }) \Big). \]
We let $\bar{w}_{2, n }(x)=w_{2, n }(s_{ n } x+x_{ n })+2\ln s_{ n }$  and $\bar{u}_{1, n }(x)=u_{1, n }(s_{ n } x+x_{ n })$ for $|x|<\frac{d}{2s_{ n }}$.
Then $\bar{w}_{2,n }$ satisfies
\begin{equation*}
\Delta \bar{w}_{2,n}+ e^{\bar{w}_{2,n}}\left\{1-\frac{\varepsilon_n^2}{s_n^2} e^{\bar{w}_{2,n}} - \sigma_n (1-e^{\bar{u}_{1,n}})\right\}+\frac{\sigma_n^2s_n^2}{\varepsilon_n^2} e^{\bar{u}_{1,n}}\left(1-\frac{\varepsilon_n^2}{s_n^2} e^{\bar{w}_{2,n}}\right)
-\frac{\sigma_ns_n^2}{\varepsilon_n^2}e^{\bar{u}_{1,n}}(1-e^{\bar{u}_{1,n}})
 =0 \quad\mbox{in }~ B_{\frac{d}{2s_{ n }}}(0).
\end{equation*}
Since ${u}_{2,n}$ satisfies  Proposition \ref{alternatives}-(2b), we see that
\begin{equation*}
\frac{\varepsilon_n^2}{s_{ n }^2}=\exp(w_{2, n }(x_{ n })+2\ln\varepsilon_n)\le e^{-\nu_0}<1,
\end{equation*} and thus there is a constant $c_0\in[0,1)$ satisfying $\lim_{n\to\infty}\frac{\varepsilon_n^2}{s_{ n }^2}=c_0^2$. \\
Moreover, since $\bar{w}_{2, n }(x)=w_{2, n }(s_{ n } x+x_{ n })+2\ln s_{ n }\le \bar{w}_{2, n }(0)=w_{2, n }(x_{ n })+2\ln s_{ n }=0$
for $|x|<\frac{d}{2s_{ n }}$, it follows from Lemma \ref{cor1} and \eqref{assume4} that
\begin{equation*}
\begin{aligned}
\left\{
 \begin{array}{ll}
 0\le |\Delta\bar{w}_{2, n }|\le  {2}  \quad \mbox{in }~  B_{d/(2s_{ n })}(0), \\ [1ex]
\left\|e^{\bar{w}_{2,n}}\left(1-\frac{\varepsilon_n^2}{s_n^2} e^{\bar{w}_{2,n}}  \right)   \right\|_{L^1(B_{d/(2s_{ n })}(0))} \le 2\mathfrak{N},
 \end{array}\right.
\end{aligned}
\end{equation*}
%Green's representation formula shows that $|\nabla \bar{w}_{2, n }(x)| \le C$ uniformly for
%$|x|\le d/(2s_{ n })$.
By  Lemma \ref{harnarkineq},
$\{\bar{w}_{2, n }\}$ is bounded in $C^0_{\textrm{loc}}(B_{d/(2s_{ n })}(0))$.
There is  a function $w_*$ satisfying, passing to a subsequence,
$\bar{w}_{2, n }\to w_*$ in $C^2_{\textrm{loc}}(\mathbb{R}^2)$ as $n\to\infty$, where $w_*$ satisfies
\begin{equation*}
\begin{aligned}
\left\{
 \begin{array}{ll}
\Delta w_*+ e^{w_*}(1-c^2_0e^{w_*}) =0
 \quad \mbox{in }~ \mathbb{R}^2, \\ [1ex]
 \int_{\mathbb{R}^2} e^{w_*} |1-c^2_0e^{w_*}|  dx \le {2\mathfrak{N}},
 \quad\mbox{and}\quad w_*\le 0 \quad\mbox{in }~ \mathbb{R}^2.
 \end{array}\right.
\end{aligned}
\end{equation*}
If $c_0$=0, then we see that $\int_{\mathbb{R}^2}   e^{w_*}dx =8\pi$ from \cite{ChL1}.\\
If $c_0>0$, then we consider the function $\phi(x)=w_*(c_0x)+2\ln c_0$. It is obvious that
\begin{equation*}
\begin{aligned}
\left\{
 \begin{array}{ll}
\Delta \phi+  e^{\phi}(1-e^{\phi}) =0 \quad \mbox{in }~ \mathbb{R}^2, \\ [1ex]
 \phi\le 2\ln c_0<0, \quad\mbox{and}\quad 0<\int_{\mathbb{R}^2} e^{\phi}(1-e^{\phi}) = dx \le {2\mathfrak{N}}.
 \end{array}\right.
\end{aligned}
\end{equation*}
In view of Lemma \ref{propertyofentiresolution}, we also get that
\[ \int_{\mathbb{R}^2}  e^{\phi}(1-e^{\phi})  dx \ge 8\pi. \]
From Fatou's lemma,  we obtain  \eqref{massofsingularpoint2}. However, it  contradicts  the estimation    \eqref{w2main_ieq}, and thus we can exclude  the blow up phenomena for $w_{2,n}$, that is, \textit{(case 2)}.  Now we complete the proof of Lemma  \ref{2ndalternatives}.
\end{proof}
 In view of   Proposition \ref{alternatives}-(2a) and  Lemma \ref{2ndalternatives}, we can improve the gradient estiation of ${u}_{2,n}$ in Lemma \ref{grad_sol}.
\begin{lemma}\label{2nd_grad}Let $({u}_{1,n},{u}_{2,n})$ be a sequence of solutions for  \eqref{main_eq} over $\mathbb{T}$. Then there exists a constant $C>0$, independent of ${n}\ge1$, such that
$\left\|\nabla {u}_{2,n} \right\|_{L^\infty(\mathbb{T})}\le  {C}.$\end{lemma}
\begin{proof}By  the Green's representation formula, we have that
\[u_{2,n}(x)-\int_\mathbb{T} u_{2,n}dy=\int_\mathbb{T} \frac{1}{\varepsilon_n^2} \left\{e^{{u}_{2,n}}(1-e^{{u}_{2,n}})+\sigma_n^2 e^{{u}_{1,n}}(1-e^{{u}_{2,n}})-\sigma_n (e^{{u}_{1,n}}+e^{{u}_{2,n}})(1-e^{{u}_{1,n}})\right\}G(x,y)dy.
\]
In view of  Lemma \ref{lemma_RT},  Lemma \ref{cor1}, and \eqref{assume4}, we see that there are constants $c_0, c_1>0$, independent of $n\ge1$, satisfying
\begin{equation}\begin{aligned}\label{gradu21}
 \left|\nabla_x {u}_{2,n} (x)\right|
 &\le \left|\int_{B_d(x)}\frac{1}{\varepsilon_n^2}e^{{u}_{2,n}(y)}\left(1-e^{{u}_{2,n}}\right)\left(-\frac{x-y}{2\pi|x-y|^2}+\nabla \gamma(x,y)\right)dy\right|+c_0+2\mathfrak{N}+ \mathfrak{N}^2\varepsilon_n^2\\
&\leq c_1 \left( \frac{  \left\|e^{{u}_{2,n}}\left(1-e^{{u}_{2,n}}\right)\right\|_{L^\infty(\mathbb{T})}}{ \varepsilon_n^2} +1\right).
\end{aligned}\end{equation}  In view of   Proposition \ref{alternatives}-(2a) and  Lemma \ref{2ndalternatives}, we note that $u_{2,n}$ satisfies
\begin{equation}\begin{aligned}\label{gradu212}\textrm{either}\ \  \|{u}_{2,n}\|_{L^\infty(\mathbb{T})}=O(\varepsilon_n^2),\ \  \textrm{or}\  \ \sup_{\mathbb{T}}\left({u}_{2,n}   -2\ln\varepsilon_n\right)\to-\infty\ \ \textrm{as}\  \ n\to\infty.\end{aligned}\end{equation}
By using \eqref{gradu21} and \eqref{gradu212}, we can obtain Lemma \ref{2nd_grad}.
\end{proof}

Now we consider the asymptotic behavior of $u_{1,n}$ satisfying Proposition \ref{alternatives}-(1b).
\begin{lemma}\label{1stalternatives}
Let $({u}_{1,n}, \nl)$ be a sequence of solutions of \eqref{main_eq}. \\
If   $u_{1,n}$ satisfies   Proposition \ref{alternatives}-(1b), then   as ${ n \to\infty}$, up to subsequences, ${u}_{1,n}$ satisfies one of the followings: \\
(i)    $\wlm-u_0\to\hat{w}$ in $C^{1}_{\textrm{loc}}(\mathbb{T})$, where $\hat{w}$ satisfies  $\Delta \hat{w} +e^{\hat{w}+u_0}=8\pi\mathfrak{M}$, or\\
(ii)  there exists a nonempty finite set $\mathfrak{B}=\{{q}_1,\cdots,{q}_k\}\subset\mathbb{T}$ and $k$-number of sequences  of points $q^{j}_{ n }\in\mathbb{T}$ such that $\lim_{ n \to\infty }q^{j}_{ n }={q}_j$  and  ${w}_{1,n}(q^{j}_{ n })\to+\infty$. Moreover, if  $q\in \mathfrak{B}$, then
\[ \liminf_{{n}\to\infty}\int_{B_d(q)}{e}^{ w_{1, n } }\left(1-\varepsilon_n^2 e^{w_{1,n}}\right) dx \ge 8\pi\ \ \textrm{for any}\ \ d>0.\]\end{lemma}
\begin{proof}
 The function $w_{1, n }$ satisfies
\begin{equation}
  \label{wmaineq}\begin{aligned}
 \Delta {w}_{1,n}+ e^{{w}_{1,n}}\left\{1-\varepsilon_n^2 e^{{w}_{1,n}} - \sigma_n (1-e^{{u}_{2,n}})\right\}+\frac{\sigma_n^2}{\varepsilon_n^2} e^{{u}_{2,n}}(1-\varepsilon_n^2 e^{{w}_{1,n}})
-\frac{\sigma_n}{\varepsilon_n^2}e^{{u}_{2,n}}(1-e^{{u}_{2,n}})
 =8\pi\sum_{i=1}^N m_{i}\delta_{p_i}.
\end{aligned}
\end{equation} By using Lemma \ref{cor1} and  $\frac{\sigma_n}{\varepsilon_n^2}\le\mathfrak{N}$, we see that
\begin{equation}\label{weqp1}\begin{aligned}
\mathfrak{N} +{8\pi\mathfrak{M}}&\ge\int_{\mathbb{T}}\frac{\sigma_n}{\varepsilon_n^2}e^{{u}_{2,n}}(1-e^{{u}_{2,n}})dx +{8\pi\mathfrak{M}}
\\&=  \int_{\mathbb{T}}   e^{{w}_{1,n}}\left\{1-\varepsilon_n^2 e^{{w}_{1,n}} - \sigma_n (1-e^{{u}_{2,n}})\right\}  +\frac{\sigma_n^2}{\varepsilon_n^2} e^{{u}_{2,n}}(1-\varepsilon_n^2 e^{{w}_{1,n}})  dx\\&\ge  \frac{(1-e^{-\nu_0})}{2}\int_{\mathbb{T}}   e^{{w}_{1,n}}dx,
\end{aligned} \end{equation}
and
\begin{equation}\label{weqp1l}\begin{aligned}
 {8\pi\mathfrak{M}}&\le\int_{\mathbb{T}}\frac{\sigma_n}{\varepsilon_n^2}e^{{u}_{2,n}}(1-e^{{u}_{2,n}})dx +{8\pi\mathfrak{M}}
\\&=  \int_{\mathbb{T}}   e^{{w}_{1,n}}\left\{1-\varepsilon_n^2 e^{{w}_{1,n}} - \sigma_n (1-e^{{u}_{2,n}})\right\}  +\frac{\sigma_n^2}{\varepsilon_n^2} e^{{u}_{2,n}}(1-\varepsilon_n^2 e^{{w}_{1,n}})  dx\\&\le   \int_{\mathbb{T}}   e^{{w}_{1,n}}dx+{\varepsilon_n^2}\mathfrak{N}^2.
\end{aligned} \end{equation}
Then we have
\begin{equation}\label{w1l}
0<4\pi\mathfrak{M}\le \|e^{\wlm}\|_{L^1(\mathbb{T})}\le\frac{2(\mathfrak{N}+8\pi\mathfrak{M})}{1- e^{-{\nu_0}}}.
\end{equation}
We consider the following two cases:

\medskip\noindent
\textit{(case 1)} $\sup_{\mathbb{T}}\wlm\le C$ for some constant $C>0$.

 In this case,  Lemma \ref{harnarkineq}  and  \eqref{w1l}  imply
$\wlm-u_0$ is uniformly bounded in $L^\infty(\mathbb{T})$. By $W^{2,2}$ estimation, we see that $\wlm-u_0$ is uniformly bounded in $C^{1,\alpha}(\mathbb{T})$ for some $\alpha\in(0,1)$. The assumption \eqref{assume4}, Proposition \ref{alternatives}-(2), and Lemma \ref{2ndalternatives}, we see that   $\wlm-u_0\to\hat{w}$ in $C^{1}_{\textrm{loc}}(\mathbb{T})$, where $\hat{w}$ satisfies
\[\Delta \hat{w} +e^{\hat{w}+u_0}=8\pi\mathfrak{M}.\]

\medskip\noindent
\textit{(case 2)} $\lim_{{ n \to\infty}}\sup_{\mathbb{T}}\wlm=+\infty.$

In  this case, we claim that if  $p\in \mathfrak{B}$, then  the following ``minimal mass" result holds.
\begin{equation}\label{massofsingularpoint} \liminf_{ n \to\infty }\int_{B_d(p)}{e}^{ w_{1, n } }\left(1-\varepsilon_n^2 e^{w_{1,n}}\right) dx \ge 8\pi\ \ \textrm{for any}\ \ d>0.\end{equation}
In order to prove the claim \eqref{massofsingularpoint}, we will apply the arguments in \cite[Lemma 4.2]{CK}.
We fix a small constant $d>0$  such that $B_d(p)\cap \mathfrak{B}=\{p\}$ and $B_d(p_i)\cap B_d(p_j)=\emptyset$ if $p_i\neq p_j \in Z$. Choose a sequence of points
$\{x_{ n }\}$ in $B_d(p)$ such that
$w_{1, n }(x_{ n }) =\sup_{x\in B_{d}(p)} w_{1, n }(x)\to\infty$ and $x_{ n }\to p$ as ${n}\to\infty$, and let
\[ s_{ n } = \exp\Big(-\frac{1}{2}w_{1, n }(x_{ n }) \Big). \]
Since ${u}_{1,n}$ satisfies    Proposition \ref{alternatives}-(1b),  there exists a constant $\nu>0$ such that
\begin{equation}
  \label{nu}
 \sup_{|x-p|\le d} (w_{1, n }(x) +2\ln\varepsilon_n) \le -\nu_0.
\end{equation}
Along a subsequence, we consider the following three cases.

{\it Case 1: $p\notin {Z}$.}\\
  In this case, we can derive the estimation \eqref{massofsingularpoint} by applying the same arguments for the proof of  \eqref{massofsingularpoint2} in Lemma \ref{2ndalternatives}.

{\it Case 2: $p\in {Z}$ and $\lim_{{n}\to\infty}\frac{|x_{ n } -p|}{s_{ n }}=\infty$.}\\
For the simplicity, we  assume that $p=0$. Note that $v_{1,n}(x)=w_{1, n }(x)- 4m_{j} \ln|x| $ is a smooth function in $B_d(0)$.
For $|x|\le\frac{| x_{n}|}{2s_{n}}$, let
\begin{equation*}
\begin{aligned}
\left\{
 \begin{array}{ll}\bar{v}_{1,n}(x)=v_{1,n}(s_{n} x+ x_{n})+2\ln s_{n}+4m_{j} \ln| x_{n}|, \\
\bar{u}_{2,n}(x)=u_{2,n}(s_{n} x+ x_{n}).
 \end{array}\right.
\end{aligned}
\end{equation*}
From Lemma \ref{cor1}, we see that  $\bar{v}_{1,n}$ satisfies
\begin{equation*}
\begin{aligned}
\left\{
 \begin{array}{ll}
 -\Delta \bar{v}_{1,n}
\\ = \left|\frac{s_{n}}{| x_{n}|}x+\frac{ x_{n}}{| x_{n}|}\right|^{4m_j}e^{\bar{v}_{1,n}}\left\{1-\frac{\varepsilon_n^2}{s_{n}^2} \left|\frac{s_{n}}{| x_{n}|}x+\frac{ x_{n}}{| x_{n}|}\right|^{4m_j}e^{\bar{v}_{1,n}} - \sigma_n (1-e^{\bar{u}_{2,n}})\right\}
\\ +\frac{s_{n}^2\sigma_n^2}{\varepsilon_n^2} e^{\bar{u}_{2,n}}\left(1-\frac{\varepsilon_n^2}{s_n^2} \left|\frac{s_{n}}{| x_{n}|}x+\frac{ x_{n}}{| x_{n}|}\right|^{4m_j}e^{\bar{v}_{1,n}}\right)
 -\frac{s_{n}^2\sigma_n}{\varepsilon_n^2}e^{\bar{u}_{2,n}}(1-e^{\bar{u}_{2,n}})
 \quad\quad\ \textrm{in }~ B_{\frac{| x_{n}|}{2s_{n}}}(0), \\ [3ex]
  \int_{B_{\frac{| x_{n}|}{2s_{n}}}(0)}\left|\frac{s_{n}}{| x_{n}|}x+\frac{ x_{n}}{| x_{n}|}\right|^{4m_j}e^{\bar{v}_{1,n}}\left|1-\frac{\varepsilon_n^2}{s_{n}^2} \left|\frac{s_{n}}{| x_{n}|}x+\frac{ x_{n}}{| x_{n}|}\right|^{4m_j}e^{\bar{v}_{1,n}}  \right| dx \le {8\pi\mathfrak{M}},
 \\ \frac{\varepsilon_n^2}{s_{n}^2} \big|\frac{s_{n}}{| x_{n}|}x +\frac{ x_{n}}{| x_{n}|}\big|^{4m_{j}}
 e^{\bar{v}_{1,n}} \le e^{-\nu_0}<1 \quad \mbox{in }~ B_{\frac{| x_{n}|}{2s_{n}}}(0).
 \end{array}\right.
\end{aligned}
\end{equation*}
Since $\bar{v}_{1,n}(0) =w_{1, n }( x_{n})+2\ln s_{n} =0$ and
\[ \bar{v}_{1,n}(x)=w_{1, n }(s_{n} x+ x_{n})+2\ln s_{n}-4m_{j}
 \ln\Big|\frac{s_{n}}{| x_{n}|}x+\frac{ x_{n}}{| x_{n}|}\Big|\le4m_{j} \ln2
 \quad \mbox{for }~ |x|<\frac{| x_{n}|}{2s_{n}}, \]
it follows from Lemma \ref{harnarkineq} that $\bar{v}_{1,n}$ is bounded in
$C^0_{\textrm{loc}}(B_{| x_{n}|/(2s_{n})}(0))$. Passing to subsequences,
we may assume that $\lim_{{n}\to\infty}\frac{\varepsilon_n^2}{s_{n}^2}=c_1^2$ for some $c_1\in[0,1)$,
$\lim_{{n}\to\infty}\frac{ x_{n}}{| x_{n}|}=y_0$ for some $y_0\in S^1$, and
$\bar{v}_{1,n}\to v^*$ in $C^2_{\textrm{loc}}(\mathbb{R}^2)$, where $v^*$ satisfies
\begin{equation*}
\begin{aligned}
\left\{
 \begin{array}{ll}
\Delta v^*+  e^{v^*}(1-c_1^2e^{v^*}) =0 \quad \textrm{in }~ \mathbb{R}^2, \\ [2ex]
 \int_{\mathbb{R}^2}|e^{v^*}(1-c_1^2e^{v^*})|dx \le {8\pi\mathfrak{M}}, \\ v^*\le 4m_{j} \ln 2 \quad\mbox{in }~ \mathbb{R}^2.
 \end{array}\right.
\end{aligned}
\end{equation*}
By the same argument for  the proof of  \eqref{massofsingularpoint2} in Lemma \ref{2ndalternatives},   \cite{ChL1} and  Lemma \ref{propertyofentiresolution} imply the estimation \eqref{massofsingularpoint}.

{\it Case 3:  $p\in {Z}$ and $\frac{| x_{n} -p|}{s_{n}}\le C$ for some constant $C>0$.}\\
As in Case 2, we assume that $p= 0$ and $w_{1, n }(x)= 4m_{j} \ln|x| +v_{1,n}(x)$ on $B_d(0)$
for a function $v_{1,n} \in C^\infty (B_d(0))$. We may assume that $B_d(0)\cap Z=\{0\}$.
For $|x|\le\frac{d}{2s_{n}}$, let
\begin{equation*}
\begin{aligned}
\left\{
 \begin{array}{ll}\hat{v}_{1,n}(x)=v_{1,n}(s_{n} x + x_{n}) +2(2m_{j}+1)\ln s_{n},
\\ \hat{u}_{2,n}(x)=u_{2,n}(s_{n} x + x_{n}).
 \end{array}\right.
\end{aligned}
\end{equation*}
Then we have
\begin{equation*}
\begin{aligned}
\left\{
 \begin{array}{ll}
 -\Delta \hat{v}_{1,n}
\\=  \left| x + \frac{x_{n}}{s_n}\right|^{4m_j}e^{\hat{v}_{1,n}}\left\{1-\frac{\varepsilon_n^2}{s_n^2} \left| x + \frac{x_{n}}{s_n}\right|^{4m_j}e^{\hat{v}_{1,n}} - \sigma_n (1-e^{\hat{u}_{2,n}})\right\}\\+\frac{s_n^2\sigma_n^2}{\varepsilon_n^2} e^{\hat{u}_{2,n}}\left(1-\frac{\varepsilon_n^2}{s_n^2} \left| x + \frac{x_{n}}{s_n}\right|^{4m_j}e^{\hat{v}_{1,n}}\right)
-\frac{s_n^2\sigma_n}{\varepsilon_n^2}e^{\hat{u}_{2,n}}(1-e^{\hat{u}_{2,n}})
 \quad \mbox{in }~ B_{\frac{d}{2s_{n}}}(0), \\ [3ex]
 \int_{|x|\le d/(2s_{n})} \left| x + \frac{x_{n}}{s_n}\right|^{4m_j}e^{\hat{v}_{1,n}}\left|1-\frac{\varepsilon_n^2}{s_n^2} \left| x + \frac{x_{n}}{s_n}\right|^{4m_j}e^{\hat{v}_{1,n}}  \right| dx \le {8\pi\mathfrak{M}}.
 \end{array}\right.
\end{aligned}
\end{equation*}
We note that
\[ \Big|x+\frac{ x_{n}}{s_{n}}\Big|^{4m_{j}}e^{\hat{v}_{1,n}}
 = s_{n}^2 e^{w_{1, n } (s_{n} x + x_{n})} \le 1 \quad\mbox{for }~ |x|\le \frac{d}{2s_{n}}, \]
and consequently
\begin{equation*}
\begin{aligned}
\hat{v}_{1,n}(x) \le -4m_{j} \ln \Big| x+\frac{ x_{n}}{s_{n}}\Big|
 \le -4m_{j} \ln (|x|-C) \quad \mbox{for }~ C<|x|<\frac{d}{2s_{n}}.
\end{aligned}
\end{equation*}
We also see that
\[ \frac{\varepsilon_n^2}{s_{n}^2}\Big|x+\frac{ x_{n}}{s_{n}}\Big|^{4m_{j}}e^{\hat{v}_{1,n}}
 \le e^{-\nu_0}<1 \quad\mbox{for }~ |x|\le \frac{d}{2s_{n}}, \]
and
\begin{equation*}
\hat{v}_{1,n}(0) = w_{1, n }( x_{n}) -4m_{j} \ln| x_{n}|+2(1+2m_{j}) \ln s_{n}
 =-4m_{j} \ln \Big(\frac{| x_{n}|}{s_{n}}\Big) \ge c,
\end{equation*}
for some constant $c\in\R$.
Hence, it follows from Lemma \ref{harnarkineq}  that $\hat{v}_{1,n}$ is bounded in
$C^0_{\textrm{loc}}(B_{d/(2s_{n})})$. Passing to subsequences,
we may assume that $\lim_{{n}\to\infty}\frac{ x_{n}}{s_{n}}=y_2\in\mathbb{R}^2$,
$\lim_{{n}\to\infty}\frac{\varepsilon_{n}}{s_n}=c_2\in [0,1)$ and
$\hat{v}_{1,n}\to v_*$ in $C^2_{\textrm{loc}}(\mathbb{R}^2)$, where $v_*$ satisfies
\begin{equation*}
\begin{aligned}
\left\{
 \begin{array}{ll}
\Delta v_*+ |x+y_2|^{4m_{j}}e^{v_*}
 (1-c_2^2|x+y_2|^{4m_{j}}e^{v_*})=0
 \quad\mbox{in }~\mathbb{R}^2, \\ [2ex]
  \int_{\mathbb{R}^2} |x+y_2|^{4m_{j}}e^{v_*}
|(1-c_2^2|x+y_2|^{4m_{j}}e^{v_*})|dx \le {8\pi\mathfrak{M}},
 ~\quad \sup_{x\in\mathbb{R}^2}|x+y_2|^{4m_{j}}e^{v_*}\le1.
 \end{array}\right.
\end{aligned}
\end{equation*}
Letting $u_*(x)=v_*(x)+4m_{j} \ln|x+y_2|$, we have
\begin{equation*}
\Delta u_*+ e^{u_*}(1-c_2^2e^{u_*})
=8\pi m_{j}\delta_{-y_2} \ \ \textrm{in}\ \ \mathbb{R}^2.
\end{equation*}
If $c_2=0$, then we see that $\int_{\mathbb{R}^2}   e^{u_*} dx\ge 8\pi(1+2m_{j})$ (\cite{ChL2, PrT}).
If $c_2>0$,  then we consider the function ${u}_{1,n}(x)=u_*(c_2x)+2\ln c_2$. Then $u$ satisfies
\begin{equation*}
\begin{aligned}
\left\{
\begin{array}{ll}
 \Delta u+  e^{u}(1-e^{u})  =8\pi m_{j}\delta_{-y_2/c_2} \quad\mbox{in }~ \mathbb{R}^2, \\ [1ex]
 \int_{\mathbb{R}^2} | e^{u}(1-e^{u}) | dx \le {8\pi\mathfrak{M}}, ~\quad \sup_{\mathbb{R}^2}u<-\nu_0 \quad \mbox{in }~ \mathbb{R}^2.
 \end{array}\right.
\end{aligned}
\end{equation*}
It follows from Lemma \ref{lemma2.1} that
\[ \int_{\mathbb{R}^2}  e^{u}(1-e^{u}) dx > 8\pi (1 +2m_{j}), \]
and Fatou's lemma implies that $\displaystyle\liminf_{{n}\to\infty} \int_{B_d(q)}
 e^{w_{1, n }} (1-\e^2 e^{w_{1, n }})  dx \ge 8\pi(2m_{j}+1)$. Now we complete the proof of Lemma \ref{1stalternatives}.
\end{proof}

Now we are realy to complete the proof of       Theorem \ref{BrezisMerletypealternatives}.\\
   \textbf{Proof of Theorem \ref{BrezisMerletypealternatives}: }
Lemma \ref{1stalternatives} shows that if along a subsequence,
$\lim_{{n}\to\infty} \big(\sup_\mathbb{T} w_{1, n } \big) =+\infty$, then
$\{w_{1, n }\}$ has a nonempty finite blow-up set $\mathfrak{B}\subset \mathbb{T} $,
and $|\mathfrak{B}| \le \mathfrak{M}$. We also see that
for any compact set $K\subseteq\mathbb{T}\setminus \mathfrak{B}$,
there exists a constant $C_{K}>0$ such that
\begin{equation}
  \label{locbdd}
 \sup_{K}w_{1, n }\le C_{K}.
\end{equation}
In order to complete the proof of Theorem \ref{BrezisMerletypealternatives}, in view of  Proposition \ref{alternatives}, Lemma \ref{2ndalternatives},  and  Lemma \ref{1stalternatives},  it is enough to show that the blow up phenomena implies the concentration of mass as in \cite{BT, BM,  CK}.
   Compared  with the previous results \cite{BT, BM,  CK}, in our case,  we note that the convergence rate of ${u}_{2,n}$ in Proposition \ref{alternatives}-(2a) and  the gradient estimation in   Lemma \ref{2nd_grad} would play an important role  in the estimation for the pohozaev identity \eqref{poho} below.

Now we assume that the case (ii) in  Lemma \ref{1stalternatives} occurs, and thus     $\mathfrak{B}\neq\emptyset$.
Choose a small constant $d>0$ satisfying
\[ B_{2d}(x)\cap B_{2d}(y)=\emptyset \quad\mbox{for }~ x, y\in \mathfrak{B}
 \quad\mbox{and}\quad x\neq y. \]
For each $p\in \mathfrak{B}$, we let $\{x_{ n ,p}\}$ be a sequence of points such that
\[x_{ n ,p}\to p\in \mathfrak{B}\ \ \textrm{ and}\ \ w_{1, n }(x_{ n ,p})=\sup_{B_d(p)} w_{1, n }\to\infty\ \ \textrm{ as}\ \ {n}\to\infty.\]
Recall that ${w}_{1,n}$ satisfies \begin{equation}
  \label{wmaineq2}\begin{aligned}
 \Delta {w}_{1,n}+ e^{{w}_{1,n}}\left\{1-\varepsilon_n^2 e^{{w}_{1,n}} - \sigma_n (1-e^{{u}_{2,n}})\right\}+\frac{\sigma_n^2}{\varepsilon_n^2} e^{{u}_{2,n}}(1-\varepsilon_n^2 e^{{w}_{1,n}})
-\frac{\sigma_n}{\varepsilon_n^2}e^{{u}_{2,n}}(1-e^{{u}_{2,n}})
 =8\pi\sum_{i=1}^N m_{i}\delta_{p_i}.
\end{aligned}
\end{equation}
We shall prove that
\begin{equation}
  \label{eq:-infty}
 \lim_{{n}\to\infty}\Big(\inf_{\partial B_r(p)}(w_{1, n }-u_0)\Big)=-\infty
\end{equation}
for any $r\in(0,d]$ and all $p\in \mathfrak{B}$.
Once we get the estimation \eqref{eq:-infty}, then (\ref{locbdd}) and Lemma \ref{harnarkineq} imply that
\[ \lim_{{n}\to\infty} \big( \sup_{\mathbb{T}\setminus[ \cup_{q_j\in\mathfrak{B}} B_r(q_j)]} (w_{1, n }-u_0) \big)
 =-\infty \quad \mbox{for any }~ r\in(0,d]. \]
To prove \eqref{eq:-infty}, we argue by contradiction and
suppose that there exist $r\in(0,d]$ and $p\in \mathfrak{B}$ such that
\begin{equation*}
 \begin{aligned}
\lim_{{n}\to\infty} \Big(\inf_{\partial B_r(p)}(w_{1, n }-u_0)\Big)\ge c,
 \end{aligned}
\end{equation*}
for some constant $c\in\mathbb{R}$. For the simplicity, we assume that $p=0$.
By using (\ref{locbdd}) and Lemma \ref{harnarkineq}, we can verify that $\{w_{1, n } -u_0\}$ is bounded
in $C_{loc}^0 (B_{2d}(0) \setminus \{0\})$. Then elliptic estimates imply that
there exists a function $\xi\in C^2_{loc}(B_{2d}(0)\setminus \{0\})$ such that
along a subsequence $w_{1, n } -u_0\to\xi$ in $C^1_{loc}(B_{2d}(0) \setminus \{0\})$.
In view of Lemma \ref{lemma_RT}, we see that
\begin{equation*}
 \begin{aligned}
{e}^{ w_{1, n } }\left(1-\varepsilon_n^2 e^{w_{1,n}}\right)
\to  e^{\xi+u_0} +\alpha_p\delta_{0} \quad (\alpha_p\ge 8\pi)
 \end{aligned}
\end{equation*}
in the sense of measure on $B_{2d}(0)$.
By Lemma \ref{cor1} and Fatou's lemma, we have that $e^{\xi+u_0}\in L^1(B_{2d}(0))$.
Moreover, Green's representation formula implies that
\begin{equation*}
\begin{aligned}
\xi(x)=-\frac{\alpha_p}{2\pi} \ln|x| +\phi(x) +\eta(x),
\end{aligned}
\end{equation*}
where $\eta\in C^1(B_{\frac{d}{2}}(0))$, and
\begin{equation}
  \label{definitionofphi}
\begin{aligned}
\phi(x)=\frac{1}{2\pi} \int_{B_{d}(0)} \ln\Big(\frac{1}{|x-y|}\Big) e^{(\xi+u_0)(y)}dy.
\end{aligned}
\end{equation}
We note that
\[ \phi\ge\frac{1}{2\pi}\ln\Big(\frac{1}{2d}\Big)\|e^{\xi +u_0}\|_{L^1(B_{d}(0))}
 \quad\mbox{in }~ B_{d}(0). \]
Then we see that $e^{\xi(x)}=|x|^{-\alpha_p/2\pi} e^{\phi+\eta}\ge c|x|^{-\alpha_p/2\pi}$
for $0<|x|\le \frac{d}{2}$ and some constant $c>0$.
Then the integrability of $e^{\xi+u_0}$ implies that
\begin{equation}
  \label{contradtioneq}
4\pi(1+2m)>\alpha_p,
\end{equation}
where $m=m_{j}$ if $p=q_{j}\in \mathfrak{B}\cap {Z}$, and $m=0$ if $p\in \mathfrak{B}\setminus {Z}$.\\
Let $\phi_{1, n }(x)=w_{1, n }(x)-4m\ln|x|$. Then we have the following equation:
\begin{equation}
\begin{aligned}  \label{wisingentirernsoleq}
\left\{
 \begin{array}{ll}
\Delta\pl+\sigma_n\Delta {u}_{2,n}+\frac{(1-\sigma_n^2)}{\varepsilon_n^2}e^{{u}_{1,n}}(1-e^{{u}_{1,n}})-\frac{(1-\sigma_n^2)\sigma_n}{\varepsilon_n^2}e^{{u}_{1,n}}(1-e^{{u}_{2,n}})=0,\\
\sigma_n\Delta\pl+ \Delta {u}_{2,n}+\frac{(1-\sigma_n^2)}{\varepsilon_n^2}e^{{u}_{2,n}}(1-e^{{u}_{2,n}})-\frac{(1-\sigma_n^2)\sigma_n}{\varepsilon_n^2}e^{{u}_{2,n}}(1-e^{{u}_{1,n}})=0.
 \end{array}\right.
\end{aligned}
\end{equation}
By multiplying the first equation in  (\ref{wisingentirernsoleq}) by $x\cdot\nabla \phi_{1, n }$  and  the second equation in  (\ref{wisingentirernsoleq}) by $x\cdot\nabla {u}_{2,n}$,   we get that for any constants $c_0,\ c_1 \in\mathbb{R}$,  and $r\in\left(0,\frac{d}{2}\right)$,
\begin{equation}
\begin{aligned} \label{poho}
&\int_{\partial B_r(0)} \left[ \frac{ (x\cdot\nabla \phi_{1, n })^2}{|x|}
 -\frac{|x||\nabla \phi_{1, n }|^2}{2} +\frac{(x\cdot\nabla {u}_{2,n})^2}{|x|}
 -\frac{|x||\nabla {u}_{2,n}|^2}{2} \right] d\sigma \\
&+\sigma_n\int_{\partial B_r(0)} \left[ \frac{ 2 (x\cdot\nabla \phi_{1, n }) (x\cdot\nabla {u}_{2,n})}{|x|}
 -|x| \nabla \phi_{1, n }\cdot \nabla {u}_{2,n} \right] d\sigma \\
&+\frac{(1-\sigma_n^2)}{\varepsilon_n^2}\int_{\partial B_r(0)}  |x| \left[  e^{{u}_{1,n}}\left(1-\frac{1}{2}e^{{u}_{1,n}}\right) + e^{{u}_{2,n}}\left(1-\frac{1}{2}e^{{u}_{2,n}}-c_0\right)-\sigma_n\left(e^{{u}_{1,n}}+e^{{u}_{2,n}}-e^{{u}_{1,n}+{u}_{2,n}}-c_1\right)\right] d\sigma \\
 %&=  \frac{1}{\varepsilon_n^2}\int_{B_r(0)} 2(1-\sigma_n^2)e^{{u}_{1,n}}\left(1-\frac{ e^{{u}_{1,n}}}{2}\right)+ 4m(1-\sigma_n^2)e^{{u}_{1,n}}\left(1-  e^{{u}_{1,n}} \right)dx
%\\
 %&+ \frac{1}{\varepsilon_n^2}\int_{B_r(0)} 2(1-\sigma_n^2)\left\{e^{{u}_{2,n}}\left(1-\frac{1}{2} e^{{u}_{2,n}}\right) -c_0\right\}- 4m\sigma_n(1-\sigma_n^2)e^{{u}_{1,n} }\left(1-  e^{{u}_{2,n}}\right)dx
%\\&-\frac{1}{\varepsilon_n^2}\int_{B_r(0)} 2(1-\sigma_n^2)\sigma_n\left\{e^{{u}_{1,n}}+e^{{u}_{2,n}}-e^{{u}_{1,n}+{u}_{2,n}} -c_1\right\} dx
%\\
 &=   \int_{B_r(0)} 2(1-\sigma_n^2)e^{\pl+4m\ln|x|}\left(1-\frac{ e^{{u}_{1,n}}}{2}\right)+ 4m(1-\sigma_n^2)e^{\pl+4m\ln|x|}\left(1-  e^{{u}_{1,n}} \right)dx
\\
 &+ \frac{ 2(1-\sigma_n^2)}{\varepsilon_n^2}\int_{B_r(0)}\left\{e^{{u}_{2,n}}\left(1-\frac{1}{2} e^{{u}_{2,n}}\right) -c_0-\sigma_n\left(  e^{{u}_{2,n}}  -c_1\right)\right\}dx
\\&- (1-\sigma_n^2)\int_{B_{r}(0)}4m\sigma_ne^{\pl+4m\ln|x|}\left(1-  e^{{u}_{2,n}}\right)+2\sigma_n  e^{\pl+4m\ln|x|}\left(1 - e^{{u}_{2,n}}\right)dx.
\end{aligned}
\end{equation}
 In view of   Proposition \ref{alternatives}-(2a) and  Lemma \ref{2ndalternatives}, we note that $u_{2,n}$ satisfies
\begin{equation}\label{1}\textrm{either}\ \  \|{u}_{2,n}\|_{L^\infty(\mathbb{T})}=O(\varepsilon_n^2),\ \  \textrm{or}\  \ \sup_{\mathbb{T}}\left({u}_{2,n}   -2\ln\varepsilon_n\right)\to-\infty\ \ \textrm{as}\  \ n\to\infty.\end{equation}
Moreover, by Lemma \ref{2nd_grad}, we have  a constant $C>0$, independent of ${n}\ge1$, such that
\begin{equation}\label{2}\left\|\nabla {u}_{2,n} \right\|_{L^\infty(\mathbb{T})}\le  {C}.\end{equation}
Let $c_0=\frac{1}{2}$ and $c_1=1$ if $\|{u}_{2,n}\|_{L^\infty(\mathbb{T})}=O(\varepsilon_n^2)$ as $n\to\infty$. Otherwise, let $c_0=c_1=0$. From \eqref{1}, \eqref{2}, and $w_{1, n } -u_0\to\xi$ in $C^1_{loc}(B_{2d}(0) \setminus \{0\})$,   we see that for any  $r\in\left(0,\frac{d}{2}\right)$,  as $n\to\infty$,
\begin{equation}\label{tozeroasrzero}
\begin{aligned}
&\int_{\partial B_r(0)} \left[ \frac{ (x\cdot\nabla \phi_{1, n })^2}{|x|}
 -\frac{|x||\nabla \phi_{1, n }|^2}{2}  \right] d\sigma  +O(r^2\max_{ \partial B_r(0)}e^{{w}_{1,n}})  +O(r) +o(1) \\
& =\int_{B_{r}(0)} 2e^{\pl+4m\ln|x|  }\left(1-\frac{\varepsilon_n^2 e^{w_{1,n}}}{2}\right)+4me^{\pl+4m\ln|x|  }\left(1-\varepsilon_n^2 e^{w_{1,n}}\right)dx +o(1)\\
& \ge\int_{B_{r}(0)} ( 2+4m   )e^{\pl+4m\ln|x|  }\left(1- \varepsilon_n^2 e^{w_{1,n}}\right)dx +o(1).
\end{aligned}
\end{equation}
Let $\varphi_0(x)=(\xi+u_0)(x)-4m\ln|x|$. Then $\phi_{1, n }(x)=w_{1, n }(x)-4m\ln|x| \to \varphi_0(x)$ in $C^1_{loc}(B_{2d}(0) \setminus \{0\})$.
By letting ${n}\to\infty$ in (\ref{tozeroasrzero}), we get that   for any  $r\in\left(0,\frac{d}{2}\right)$,
\begin{equation}
  \label{finaltozeroasrzero}
\begin{aligned}
&\int_{\partial B_{r}(0)} \Big[ \frac{(x\cdot\nabla \varphi_0)^2}{|x|}
 -\frac{|x||\nabla \varphi_0|^2}{2} \Big] d\sigma +O(  r^{2+4m  }\max_{ \partial B_r(0)}e^{\varphi_0}) +O(r)
 \ge (2+4m  )\Big(\alpha_p +\int_{B_{r}(0)}e^{\xi+u_0}dx\Big).
\end{aligned}
\end{equation}
Let $s=\max\{0,\frac{\alpha_p}{2\pi}-4m\}$. Then there exists a constant $c>0$ such that
$|x|^{4m}e^{\varphi_0} =e^{\xi+u_0}\le c|x|^{-s}e^\phi$ in $B_{\frac{d}{2}}(0)$.
We note that $s\in[0,2)$ from (\ref{contradtioneq}).
In view of (\ref{definitionofphi}) and Corollary 1 in \cite{BM},
we see that $e^{|\phi|}\in L_{loc}^k(B_{d}(0))$ for any $k\in [1,\infty)$.
Since $\xi\in C^2_{loc}(B_{2d}(0)\setminus \{0\})$, we have
$|x|^{4m}e^{\varphi_0}\in L^t(B_{d}(0))$ for any $t\in(1,\frac{2}{s})$.
Then H\"{o}lder's inequality implies that $\phi\in L^\infty(B_{d}(0))$ and   there is a constant $C>0$ satisfying
\begin{equation}
  \label{s}
|x|^{4m}e^{\varphi_0} =e^{\xi+u_0}\le C|x|^{-s}\ \ \textrm{in}\ \ B_{\frac{d}{2}}(0).
\end{equation}
We note that for $|x|=r<\frac{d}{2}$,
\begin{equation*}
\begin{aligned}
 |\nabla\phi(x)|&\le  \frac{1}{2\pi} \Big[\int_{B_{d}(0)\setminus B_{r/2}(x)} \frac{e^{(\xi+u_0)(y)}}{|x-y|} dy
 +\int_{B_{d}(0)\cap B_{r/2}(x)} \frac{e^{(\xi+u_0)(y)}}{|x-y|} dy \Big].
\end{aligned}
\end{equation*}
Fix $t\in(1,\frac{2}{s})$ and choose a constant $a\in (0,\min\{1,2-s\})$ such that $\frac{at}{t-1} <2$. Then
H\"{o}lder's inequality implies that
\[ \int_{B_{d}(0)\setminus B_{r/2}(x)} \frac{e^{(\xi+u_0)(y)}}{|x-y|} dy
 \le  \int_{B_{d}(0)\setminus B_{r/2}(x)} \frac{Cr^{a-1}}{|y-x|^a}
 e^{(\xi +u_0)(y)} dy \le Cr^{a-1}. \]
Since $|x|=r$, we have $B_{r/2}(x)\subseteq\mathbb{T}\setminus B_{r/2}(0)$.
It follows from (\ref{s}) and $\xi\in C^2_{loc}(B_{2d}(0)\setminus \{0\})$ that
\[ \int_{B_{d}(0)\cap B_{r/2}(x)} \frac{e^{(\xi+u_0)(y)}}{|x-y|} dy
 \le \int_{|y-x|\le r/2} \frac{Cr^{-s}}{|y-x|} dy =O(r^{1-s}). \]
Since $a\in(0,2-s)$, we see that $|\nabla \phi(x)| =O(|x|^{a-1}+1)$ as $|x|\to 0$.
Consequently $\nabla\varphi_0(x)=-\frac{\alpha_p x}{2\pi|x|^2}+\nabla h(x)$
with $|\nabla h(x)| =O(|x|^{a-1}+1)$ as $|x|\to 0$.
Letting   $r\to 0$ in (\ref{finaltozeroasrzero}),
we obtain that $( 2+4m   )\alpha_p\le\frac{\alpha_p^2}{4\pi}$, which contradicts (\ref{contradtioneq}).

Therefore,  from Lemma \ref{harnarkineq}, we get that $w_{1, n }-u_0\to-\infty$ and
$w_{1, n }\to-\infty$ uniformly on any compact subset of $\mathbb{T}\setminus \mathfrak{B}$.
By  Lemma \ref{lemma_RT} and Lemma \ref{cor1}, along a subsequence,
${e}^{ w_{1, n } }\left(1-\varepsilon_n^2 e^{w_{1,n}}\right)$ converges to a nonnegative measure.   Since $w_{1, n }\to-\infty$
uniformly on any $K\Subset \mathbb{T}\setminus \mathfrak{B}$,  the measure is  a sum of Dirac measures, and Lemma \ref{1stalternatives} implies
that each Dirac mass should be greater than or equal to $8\pi$.
\hfill$\qed$

%For simplicity, we denote $Q_{q}$ by $Q$.

\bigskip

%\noindent{\bf Acknowledgement}\\
%Y. Lee was supported by the National Research Foundation of Korea(NRF) grant funded by the %Korea government(MSIT) (No. NRF-2018R1C1B6003403).

\end{document}